\documentclass[12pt, a4]{amsart}
\usepackage{amscd, amsmath, amssymb}
\usepackage{fullpage}

\theoremstyle{plain}
\newtheorem{thm}{Theorem}[section]

\theoremstyle{definition}
\newtheorem{defn}[thm]{Definition}
\newtheorem{ex}[thm]{Example}
\newtheorem{rem}[thm]{Remark}
\newtheorem{cor}[thm]{Corollary}
\numberwithin{equation}{section}

\newcommand{\R}{\mathbb{R}}

\begin{document}

\title[Eigenvalues of elliptic functional differential systems]{Eigenvalues of elliptic functional differential systems via a Birkhoff--Kellogg type theorem} 

\date{}

\author[G. Infante]{Gennaro Infante}
\address{Gennaro Infante, Dipartimento di Matematica e Informatica, Universit\`{a} della
Calabria, 87036 Arcavacata di Rende, Cosenza, Italy}%
\email{gennaro.infante@unical.it}%

\begin{abstract} 
Motivated by recent interest on Kirchhoff-type equations, in this short note we utilize a classical, yet very powerful, tool of nonlinear functional analysis in order to investigate the existence of positive eigenvalues of systems of elliptic functional differential equations. An example is presented to illustrate the theory.
\end{abstract}




\maketitle

\section{Introduction}
A well known result in nonlinear analysis is the Birkhoff-Kellogg invariant-direction Theorem~\cite{B-K-1922}. In the case of an infinite-dimensional normed linear space $V$ this theorem reads as follows.
\begin{thm}~\cite[Theorem~6.1]{GD}
Let $U$ be a bounded open neighborhood of $0$ in an infinite-dimensional normed linear space $(V, \| \, \|)$, and let $T:\partial U \to V$ be a compact map satisfying $\|T(x)\| \geq \alpha$ for some $\alpha > 0$ for every $x$ in $\partial U$. Then there exist $x_0 \in \partial U$ and $ \lambda_{0} \in (0,+\infty)$ such that $x_0 = \lambda_{0} F(x_0).$ 
\end{thm}
The invariant direction Theorem has been object of deep studies in the past, with applications and extensions in several directions, we refer the reader to \cite{applbook, BK, Casey, Fitz-Petr, Fitz-Petr2, Fu-Pe-Vi, Fu-Vi, guosun, Krasno, Kryszewski, Ivar-Charlie, Oregan} and references therein. In particular, we highlight that~\cite{Fitz-Petr, Krasno, Ivar-Charlie} provide interesting applications to the existence of eigenvalues and eigenfunctions of elliptic boundary value problems.

Here we make use of the following Birkhoff-Kellogg type result, which is set in cones. This is a special case of a result due to Krasnosel'ski\u{i} and Lady\v{z}enski\u{\i}~\cite{Kra-Lady}, see also~\cite[Theorem~5.5]{Krasno}.
Before stating the result, we recall that a cone $C$ of a real Banach space $(X,\| \, \|)$ is a closed set with $C+C\subset C$, $\mu C\subset C$ for all $\mu\ge 0$ and $C\cap(-C)=\{0\}$, and we introduce the following notation:
$$C_{r}:=\{x\in C: \|x\|<r\},\,
\overline{C}_{r}:=\{x\in C: \|x\|\leq r\},\ \text{and}\ \partial
C_{r}:=\{x\in C: \|x\|=r\}.$$

\begin{thm}\label{B-K}
Let $(X,\| \, \|)$ be a real Banach space, let $T:\overline{C}_{r}\to C$ be compact and suppose that 
$$\inf_{x\in \partial C_{r}}\|Tx\|>0.$$
Then there exist $\lambda_{0}\in (0,+\infty)$ and $x_{0}\in \partial C_{r}$
such that $x_{0}=\lambda_{0} Tx_{0}$.
\end{thm}

By means of Theorem~\ref{B-K} we discuss the solvability, with respect to the parameter $\lambda$, of the following system of second order elliptic functional differential equations subject to functional boundary conditions (BCs)
\begin{equation}
 \label{nellbvp-intro}
 \left\{
\begin{array}{ccc}
L_i u_i=\lambda f_i(x, u, Du, w_i [u]), & \text{in}\ \Omega, & i=1,2,\ldots,n, \\
 B_i u_i=\lambda \zeta_i (x)h_{i}[u], & \text{on}\ \partial \Omega, & i=1,2,\ldots,n,
\end{array}
\right.
\end{equation}
where $\Omega \subset \mathbb{R}^{n}$ is a bounded domain with a sufficiently smooth boundary, $L_i$ is a strongly uniformly elliptic operator, $B_i$ is a first order boundary operator, 
 $u=(u_1,\dots, u_n)$, 
 $Du=(\nabla u_1,\dots, \nabla u_n) $, $f_i$ are continuous functions, $\zeta_i$ are sufficiently regular functions, $w_{i}$ and $h_{i}$ are suitable compact functionals.

The class of systems occurring in~\eqref{nellbvp-intro} is fairly general and allows us to deal with nonlocal problems of Kirchhoff-type. 
This is a very active area of research, a typical example of a Kirchhoff-type problem is
\begin{equation} \label{nonlocal-ToFu}
-M\left(\int_{\Omega} |\nabla u|^2\,dx\right)\Delta u= f(x,u), \ x\in \Omega,\quad u=0\ \text{on}\ \partial \Omega,
\end{equation}
which has been investigated by Ma in his survey~\cite{ToFu}. An extension to systems of the BVP~\eqref{nonlocal-ToFu} has been considered by Figueiredo and Su\'{a}rez~\cite{Giovany}, namely
\begin{equation} \label{SysGio}
\left\{
\begin{array}{ll}
-M_1\left(\int_{\Omega} |\nabla u_1|^2\,dx\right)\Delta u_1= f_1(x,u_1,u_2),& x \text{ in }\Omega , \\
-M_2\left(\int_{\Omega} |\nabla u_2|^2\,dx\right)\Delta u_2= f_2(x,u_1,u_2), & x \text{ in }\Omega , \\
u_1=u_2=0 & \text{on }\partial \Omega .%
\end{array}%
\right. 
\end{equation}%
The approach employed in~\cite{Giovany} is the sub-supersolution method. The system~\eqref{SysGio} has been studied also
by the sub-supersolution method in~\cite{BoGuChAl, MeBoGuAl, BoBoDj, BoGu}, while variational methods were employed in~\cite{FurOliSi, LouQin, Nguyen, ZhangSun}.

Note that there has been also interest in Kirchhoff-type problems with gradient terms appearing within the nonlinearities, we mention the recent papers by 
Alves and Boudjeriou~\cite{AlBo}, Yan and co-authors~\cite{YaOrAg}, Chen~\cite{Chen} and references therein.

The framework of~\eqref{nellbvp-intro} allows us to deal with non-homogenous BCs of functional type. In the case of nonlocal elliptic equations, non-homogeneous BCs have been investigated by Wang and An~\cite{WangAn}, Morbach and Corr\~{e}a~\cite{MorCor} and by the author~\cite{gi-jepeq}. 
The formulation of the functionals occurring in~\eqref{nellbvp-intro} allows us to consider \emph{multi-point} or \emph{integral} BCs. 
There exists a wide literature on this topic, 
 we refer the reader to the reviews~\cite{Cabada1, Conti, rma, sotiris, Stik, Whyburn} and the papers~\cite{Goodrich3, Goodrich4, kttmna, ktejde, Pao-Wang, Picone, jw-gi-jlms}.

Here we discuss, under fairly general conditions, the existence of positive eigenvalues with corresponding non-negative eigenfunctions for the system~\eqref{nellbvp-intro} and illustrate how these results can be applied in the case of nonlocal elliptic systems. Our results are new and complement previous results of the author~\cite{gi-jepeq}, by allowing the presence of gradient terms within the nonlinearities and the functionals. The results also complement the ones in~\cite{sBaCgI}, by considering more general nonlocal elliptic systems.

\section{Eigenvalues and eigenfunctions}
In what follows, for every $\hat{\mu}\in (0,1)$ we denote by $C^{\hat{\mu}}(\overline{\Omega})$ the space of all $\hat{\mu}$-H\"{o}lder continuous functions $g:\overline{\Omega}\to \mathbb{R}$ and, for every $k\in \mathbb{N}$, we denote by $C^{k+\hat{\mu}}(\overline{\Omega})$ the space of all functions $g\in C^{k}(\overline{\Omega})$ such that all the partial derivatives of $g$ of order $k$ are $\hat{\mu}$-H\"{o}lder continuous in $\overline{\Omega}$ (for more details see \cite[Examples~1.13 and~1.14]{Amann-rev}).

We make the following assumptions on the domain $\Omega$ and the operators $L_i$ and $B_i$ and the functions $\zeta_i$ that occur in~\eqref{nellbvp-intro}
 (see \cite[Section~4 of Chapter~1]{Amann-rev} and \cite{Lan1, Lan2})): 
\begin{enumerate}
\item $\Omega\subset \R^m$, $m\ge 2$, is a bounded domain such that its boundary $\partial \Omega$ is an $(m-1)$-dimensional $C^{2+\hat{\mu}}-$manifold for some $\hat{\mu}\in (0,1)$, such that $\Omega$ lies locally on one side of $\partial \Omega$ (see \cite[Section 6.2]{zeidler} for more details).
\item $L_i$ is a the second order elliptic operator given by
\begin{equation*}
L_i u(x)=-\sum_{j,l=1}^m a_{ijl}(x)\frac{\partial^2 u}{\partial x_j \partial x_l}(x)+\sum_{j=1}^m a_{ij}(x) \frac{\partial u}{\partial x_j} (x)+a_i (x)u(x), \quad \mbox{for $x\in \Omega$,}
\end{equation*}
where $a_{ijl},a_{ij},a_i\in C^{\hat{\mu}}(\overline{\Omega})$ for $j,l=1,2,\ldots,m$, $a_i(x)\ge 0$ on $\bar{\Omega}$, $a_{ijl}(x)=a_{ilj}(x)$ on $\bar{\Omega}$ for $j,l=1,2,\ldots,m$. Moreover $L_i$ is strongly uniformly elliptic; that is, there exists $\bar{\mu}_{i0}>0$ such that 
$$\sum_{j,l=1}^m a_{ijl}(x)\xi_j \xi_l\ge \bar{\mu}_{i0} \|\xi\|^2 \quad \mbox{for $x\in \Omega$ and $\xi=(\xi_1,\xi_2,\ldots,\xi_m)\in\R^m$.}$$

\item $B_i$ is a boundary operator given by
$$B_i u(x)=b_i(x)u(x)+\delta_i \frac{\partial u}{\partial \nu}(x) \quad \mbox{for $x\in\partial \Omega$},$$
where $\nu$ is an outward pointing and nowhere tangent vector field on $\partial \Omega$ of class $C^{1+\hat{\mu}}$ (not necessarily a unit vector field), $\frac{\partial u}{\partial \nu}$ is the directional derivative of $u$ with respect to $\nu$, $b_i:\partial \Omega \to \R$ is of class $C^{1+\hat{\mu}}$ and moreover one of the following conditions holds:
\begin{enumerate}
\item $\delta_i =0$ and $b_i(x)\equiv 1$ (Dirichlet boundary operator).
\item $\delta_i =1$, $b_i(x)\equiv 0$ and $a_i(x)\not\equiv 0$ (Neumann boundary operator).
\item $\delta_i =1$, $b_i(x)\ge 0$ and $b_i(x)\not\equiv 0$ (Regular oblique derivative boundary operator).
\end{enumerate}

\item $\zeta_i \in C^{2-\delta_i+\hat{\mu}}(\partial\Omega)$. 
\end{enumerate}

It is known that, under the previous conditions (see \cite{Amann-rev}, Section 4 of Chapter 1), a strong maximum principle holds,
given $g\in C^{\hat{\mu}}(\bar{\Omega})$, the~BVP 
 \begin{equation}
 \label{eqelliptic}
 \left\{
\begin{array}{ll}
  L_{i}u(x)=g(x), & x\in \Omega, \\
 B_{i}u(x)=0, & x\in \partial \Omega,
\end{array}
\right.
\end{equation}
admits a unique classical solution $u\in C^{2+\hat{\mu}}(\bar{\Omega})$ and, moreover, given $\zeta_i \in C^{2-\delta_i+\hat{\mu}}(\partial\Omega)$ the~BVP 
 \begin{equation}
 \label{eqellipticbc}
 \left\{
\begin{array}{ll}
  L_{i}u(x)=0, & x\in \Omega, \\
 B_{i}u(x)=\zeta_i(x), & x\in \partial \Omega,
\end{array}
\right.
\end{equation}
also admits a unique solution $\gamma_i \in C^{2+\hat{\mu}}(\bar{\Omega})$.

In order to investigate the solvability of the system~\eqref{nellbvp-intro}, we make use of the cone of non-negative functions $\hat{P}=C(\bar{\Omega},\R_+)$. The solution operator associated to the BVP~\eqref{eqelliptic}, $K_i:C^{\hat{\mu}}(\bar{\Omega})\to C^{2+\hat{\mu}}(\bar{\Omega})$, defined as $K_{i}g=u$ is linear and continuous. It is also known (see \cite{Amann-rev}, Section 4 of Chapter 1) that $K_{i}$ can be extended uniquely to a continuous, linear and compact operator (that we denote again by the same name) $K_{i}:C(\bar{\Omega})\to C^{1}(\bar{\Omega})$ that leaves the cone $\hat{P}$ invariant, that is $K_{i}(\hat{P})\subset \hat{P}$.

We utilize the space $C^{1}(\bar{\Omega},\R^n)$, endowed with the norm 
$$\|u\|_{1}:=\displaystyle\max \{\|u_i\|_{\infty}, \|\partial_{x_j} u_i \|_{\infty}: i=1,2,\ldots,n,\ j=1,2,\ldots,m\},$$ where $\|z\|_{\infty}=\displaystyle\max_{x\in \bar{\Omega}}|z(x)|,$ and consider the cone $P=C^{1}(\bar{\Omega},\R^n_+)$. We rewrite the elliptic system~\eqref{nellbvp-intro} as a fixed point problem, by considering the operators $T,\Gamma:C^1(\bar{\Omega}, \R^n) \to C^1(\bar{\Omega},\R^n)$ given by 
\begin{align}
T(u):=(K_i F_i(u))_{i=1..n},\quad
\Gamma (u):=(\gamma_i h_{i}[u] )_{i=1..n},
\end{align}
where $K_i$ is the above mentioned extension of the solution operator associated to~\eqref{eqelliptic},
 $\gamma_i \in C^{2+\hat{\mu}}(\overline{\Omega})$ is the unique solution of the BVP~\eqref{eqellipticbc} and
$$F_i(u)(x):=f_i(x,u(x),Du(x), w_i[u]),\ \text{for}\ u\in C^{1}(\bar{\Omega}, I)\ \text{and}\ x\in \bar{\Omega}.$$

\begin{defn} We say that $\lambda$ is an~\emph{eigenvalue} of the system~\eqref{nellbvp-intro} if there exists 
 $u\in C^1(\bar{\Omega})$ with $\|u\|_1>0 $ such that the pair $(u,\lambda)$ satisfies the operator equation 
\begin{equation}\label{chareq}
u=\lambda(Tu+\Gamma u)=\lambda( K_i F_i(u)+ \gamma_i h_{i}[u] )_{i=1..n}.
\end{equation}
If the pair $(u,\lambda)$ satisfies~\eqref{chareq} we say that $u$ is an~\emph{eigenfunction} of the system~\eqref{nellbvp-intro} corresponding to the eigenvalue $\lambda$. If, furthermore, the components of $u$ are non-negative, we say that $u$ is a \emph{non-negative eigenfunction} of the system~\eqref{nellbvp-intro}.
\end{defn}

\begin{thm}\label{eigen}Let $\rho \in (0,+\infty)$ and assume the following conditions hold. 
\begin{itemize}

\item[(a)] For every $i=1,2,\ldots,n$, $w_{i}: \overline{P}_{\rho} \to \mathbb{R}$ is continuous and there exist
$\underline{w}_{i,\rho}, \overline{w}_{i,\rho}\in \R$ such that
\begin{equation}\label{w-cond}
\underline{w}_{i,\rho}\leq w_{i}[u]\leq \overline{w}_{i,\rho},\ \text{for every}\ u\in \overline{P}_{\rho}.
\end{equation}

\item[(b)] 
For every $i=1,2,\ldots,n$, $f_i\in C(\Pi_{\rho}, \R)$ and there exist $\delta_i \in C(\bar{\Omega},\R_+)$ such that
\begin{equation}\label{f-cond}
f_{i}(x,u,v,w)\ge \delta_{i, \rho}(x),\ \text{for every}\ (x,u,v,w)\in \Pi_{\rho},
\end{equation}
where $$\Pi_{\rho}:=\bar{\Omega}\times [0,\rho]^n \times [-\rho,\rho]^{m\times n}\times [\underline{w}_{i,\rho}, \overline{w}_{i,\rho}].$$

\item[(c)] For every $i=1,2,\ldots,n$, $\zeta_i \in C^{2-\delta_i+\hat{\mu}}(\partial\Omega)$, $\zeta_i\geq 0$, 
$h_{i}: \overline{P}_{\rho} \to \R$ is continuous and bounded. Let $\eta_{i, \rho}\in [0,+\infty)$ be such that 
\begin{equation}\label{h-cond}
h_{i}[u]\geq \eta_{i, \rho},\ \text{for every}\ u\in \overline{P}_{\rho} .
\end{equation}

\item[(d)] There exist $i_0\in \{1,\ldots,n\}$ and $ \phi_{i_0, \rho} \in (0,+\infty)$ such that
\begin{equation}\label{c-cond}
\|K_{i_0}(\delta_{i_0, \rho})+ \eta_{i_0, \rho} \gamma_{i_0}\|_{\infty}\geq \phi_{i_0, \rho}.
\end{equation}
\end{itemize}

Then the system~\eqref{nellbvp-intro} has a positive eigenvalue with an associated eigenfunction $u\in \partial P_{\rho}$.
\end{thm}
\begin{proof}
Due to the assumptions above, the operator $T+\Gamma$ maps $\overline{P}_{\rho}$ into $P$ and is compact (by construction, the map $F$ is continuous and bounded and $\Gamma$ is a finite rank operator). Take $u\in \partial P_{\rho}$, then for every $x\in \bar{\Omega}$ we have
\begin{equation}\label{lwest}
K_{i_0} F_{i_0} u(x)+ \gamma_{i_0}(x) h_{i_0}[u] 
\ge  K_{i_0}(\delta_{i_0, \rho})(x) + \eta_{i_0, \rho} \gamma_{i_0}(x) .
\end{equation}
Taking the supremum for $x\in \bar{\Omega}$ in~\eqref{lwest} we obtain 
\begin{equation}\label{lwest2}
\| Tu+\Gamma u\|_{1}\geq \|T_{i_0}u+\Gamma_{i_0} u\|_{\infty}\geq \|K_{i_0}(\delta_{i_0, \rho})+ \eta_{i_0, \rho} \gamma_{i_0}\|_{\infty}\geq \phi_{i_0, \rho}.
\end{equation}
Note that the RHS of \eqref{lwest2} does not depend on the particular $u$ chosen. Therefore we have
$$
\inf_{u\in \partial P_{\rho}}\| Tu+\Gamma u\|\geq \phi_{i_0, \rho}>0,
$$
and the result follows by Theorem~\ref{B-K}.
\end{proof}
\begin{rem}
Note that we have chosen to use \emph{inequalities} in~\eqref{w-cond}-\eqref{c-cond}; this is due that, in applications, it is often easier and somewhat more efficient to use~\emph{estimates} on the nonlienarieties involved. Furthermore note that, in our reasoning, what really matters is that some positivity occurs in one component of the system, either in the nonlinearity $f_i$ or in the functional $h_i$.
\end{rem}
The following Corollary provides a sufficient condition for the existence of an unbounded set of eigenfunctions for the system~\eqref{nellbvp-intro}.
\begin{cor}\label{CorEig}
In addition to the hypotheses of Theorem~\ref{eigen}, assume that $\rho$ can be chosen arbitrarily in $(0,+\infty)$. Then for every $\rho$ there exists a non-negative eigenfunction $u_{\rho}\in \partial P_{\rho}$ of the system~\eqref{nellbvp-intro} to which corresponds a $\lambda_{\rho} \in (0,+\infty)$.
\end{cor}

We now show the applicability of the above results in the context of systems of nonlocal elliptic equations with functional BCs.
\begin{ex}
Take $\Omega=\{x\in \mathbb{R}^2 : \|x\|_2<1\}$ and consider the system
\begin{equation} \label{example}
\left\{
\begin{array}{ll}
-\bigl(e^{u_2(0)}+\int_{\Omega}| \nabla u_1|^2\,dx\bigr)\Delta u_1=\lambda e^{u_1}(1+| \nabla u_2|^2),& \text{in }\Omega , \\
-e^{(\int_{\Omega}| \nabla u_1|^2+| \nabla u_2|^2\,dx)}\Delta u_2=\lambda u_2^2| \nabla u_1|^2, & \text{in }\Omega , \\
u_1=\lambda h_{1}[(u_1,u_2)],\ u_2=\lambda h_{2}[(u_1,u_2)], & \text{on }\partial \Omega ,%
\end{array}%
\right. 
\end{equation}%
where 
$$
h_{1}[(u_1,u_2)]= u_1(0)+\bigl( \frac{\partial u_2}{\partial x_1}(0)\bigr)^{2}\ \mbox{and} \ h_{2}[(u_1,u_2)]=(u_1(0))^2+\int_{\Omega}| \nabla u_2|^2\,dx.
$$
Denote by $\hat{1}$ the function equal to $1$ on $\bar{\Omega}$.
Note that for $i=1,2$, $K_{i}(\hat{1})=\frac{1}{4}(1-x_1^2-x_2^2)$, where $x=(x_1,x_2)$, and $\|K_{i}(\hat{1})\|_{\infty}=\frac{1}{4}$. Furthermore note that we may take $\gamma_{1}=\gamma_{2}\equiv 1$.

 We fix $\rho\in (0,+\infty)$ and consider 
\begin{align*}
f_1(u_1,u_2, \nabla u_1, \nabla u_2, w_1[(u_1,u_2)]):&=e^{u_1}(1+| \nabla u_2|^2)w_1[(u_1,u_2)],\\
f_2(u_1,u_2, \nabla u_1, \nabla u_2, w_2[(u_1,u_2)]) :&= u_2^2| \nabla u_1|^2w_2[(u_1,u_2)],
\end{align*}
where
\begin{align*}
w_1[(u_1,u_2)]:&=\Bigl(e^{u_2(0)}+\int_{\Omega}| \nabla u_1|^2\,dx\Bigr)^{-1}, \\ 
w_2[(u_1,u_2)]:&=e^{-(\int_{\Omega}| \nabla u_1|^2+| \nabla u_2|^2\,dx)}.
\end{align*}

In this case we may take 
$$[\underline{w}_{1,\rho}, \overline{w}_{1,\rho} ]=[(2\pi \rho^2+e^{\rho})^{-1},1],\ [\underline{w}_{2,\rho}, \overline{w}_{2,\rho} ]=[e^{-4\pi \rho^2},1],$$
$$\delta_{1, \rho}(x)\equiv (2\pi \rho^2+e^{\rho})^{-1},\ \delta_{2, \rho}(x)\equiv 0,\ \eta_{1, \rho}=\eta_{2, \rho}=0,$$
and therefore we get
$$
\|K_{1}(\delta_{1, \rho})+ \eta_{1, \rho} \gamma_{1}\|_{\infty}=(8\pi \rho^2+4e^{\rho})^{-1}= \phi_{1, \rho}>0.
$$
Thus we can apply Corollary~\ref{CorEig}, obtaining a uncountably many pairs  $(u_{\rho}, \lambda_{\rho})$ of  non-negative eigenfunctions and positive eigenvalues for the system~\eqref{example}.
\end{ex}

\section*{Acknowledgement}
G. Infante was partially supported by G.N.A.M.P.A. - INdAM (Italy).

\end{document}